\documentclass[a4paper]{amsart}

\usepackage{amsmath, amsthm, amssymb}
\usepackage{mathscinet}
\usepackage{xspace}
\usepackage{xcolor}
\usepackage{tikz, pgf}
\usepackage{mathrsfs}
\usetikzlibrary{arrows}
\usetikzlibrary{arrows.meta}
\usepackage{pgfplots}
\usepgfplotslibrary{external}
\usepackage{graphicx}
\usepackage[normalem]{ulem}
\usepackage{booktabs}

\usepackage[style=alphabetic,
sorting=nty,
backend=biber,
maxbibnames=99,
doi=false,
isbn=false,
backref=false]{biblatex}
\newcommand\mptopcom{\texttt{MPTOPCOM}\xspace}
\newcommand\topcom{\texttt{TOPCOM}\xspace}
\newcommand\polymake{\texttt{polymake}\xspace}
\newcommand\mts{\texttt{mts}\xspace}

\def\R{\mathbb{R}}
\newcommand{\face}{\sigma}
\newcommand{\triang}{T}
\DeclareMathOperator{\conv}{conv}
\DeclareMathOperator{\aff}{aff}
\DeclareMathOperator{\MC}{MC}
\DeclareMathOperator{\vol}{vol}

\def\CC{\mathbb C}
\def\PP{\mathbb P}
\def\RR{\mathbb R}
\def\ZZ{\mathbb Z}

\newcommand{\tDt}{\mathcal{A}}
\newcommand{\AD}{\Delta_{\tDt}}
\newcommand{\Sfour}{\text{S}_4}

\newcommand{\twoDeltaTwoTriang}[1]{
	\coordinate (A-00) at (0,0);
	\coordinate (A-01) at (0,1);
	\coordinate (A-02) at (0,2);
	\coordinate (A-10) at (1,0);
	\coordinate (A-11) at (1,1);
	\coordinate (A-20) at (2,0);			
	\draw (A-00) -- (A-20) -- (A-02) -- cycle;
	\foreach \i/\j in {#1}{
		\draw (A-\i)--(A-\j);
	}
	\begin{scriptsize}
		\draw [fill=violet!20] (A-00) circle (2.5pt);
		\draw [fill=violet!20] (A-02) circle (2.5pt);
		\draw [fill=violet!20] (A-01) circle (2.5pt);
		\draw [fill=violet!20] (A-10) circle (2.5pt);
		\draw [fill=violet!20] (A-20) circle (2.5pt);
		\draw [fill=violet!20] (A-11) circle (2.5pt);
	\end{scriptsize}
}

\definecolor{noteColor}{rgb}{0,0.6,0.5}
\definecolor{mygreen}{rgb}{0,0.5,0.0}

\DeclareMathOperator{\gkz}{GKZ}
\newcommand{\group}{G}

\title{The Newton polytope of the discriminant of a quaternary cubic form}

\thanks{Research by the first author is supported by DFG via SFB-TRR 195: ``Symbolic Tools in Mathematics and their Application''. The second author is supported by DFG via SFB-TRR 109: ``Discretization in Geometry and Dynamics''.}


\author{Lars Kastner}
\author{Robert L\"owe}
\address{%
Institut f{\"u}r Mathematik\\
TU Berlin\\
Str.\ des 17. Juni 136, 10623 Berlin, Germany
}
\email{\{kastner,loewe\}@math.tu-berlin.de}

\newcounter{thmcounter}[section]

\newtheorem{que}[thmcounter]{Question}
\newtheorem{thm}[thmcounter]{Theorem}
\newtheorem{dfn}[thmcounter]{Definition}
\newtheorem{exa}[thmcounter]{Example}
\newtheorem{lemma}[thmcounter]{Lemma}
\newtheorem{prop}[thmcounter]{Proposition}

\usepackage{hyperref}

\begin{document}

\begin{abstract}
We determine the $166\,104$ extremal monomials of the discriminant
of a quaternary cubic form.
These are in bijection with $D$-equivalence classes of regular triangulations of the $3$-dilated tetrahedron.
We describe how to compute these triangulations and their $D$-equivalence classes in order to arrive at our main result. 
The computation poses several challenges, such as dealing with the sheer amount of triangulations effectively, as well as devising a suitably fast algorithm for computation of a $D$-equivalence class.
\end{abstract}

\maketitle


\section{Introduction}
The $A$-discriminant $\Delta_A$ is a homogeneous polynomial associated to an integral point configuration $A \subset \ZZ_{\geq 0}^d$ that detects singular polynomials supported on $A$.
The most popular example is the case of a quadratic polynomial $f=ax^2+bx+c$ supported on $A=\{0,1,2\}$, in which case the $A$-discriminant is $\Delta_A=b^2-4ac$.
Here the discriminant vanishes if and only if $f$ has a double root.

In general, $\Delta_A$ can be defined as follows.
Let $f = \sum_{a \in A} c_a \mathbf{x}^a$ be a polynomial supported on~$A$.
Then the associated hypersurface $V(f)$ is either singular or smooth, depending on the choice of coefficients $c \in \CC^A$.
The subset of $\CC^A$ consisting of those coefficients that define singular $A$-hypersurfaces has the structure of an affine cone over a projective hypersurface $\nabla_A \subset \PP(\CC^{A})$, called the \emph{discriminantal hypersurface} of~$A$.
Now the \emph{$A$-discriminant} $\Delta_A \in \CC [c_a ~|~ a\in A]$ is defined as the irreducible integral polynomial defining the discriminantal hypersurface, i.e. $V(\Delta_A) = \nabla_{A}$.
In this way $\Delta_A$ is defined uniquely up to sign.

Except for a few special cases, it is quite cumbersome to write down $A$-discriminants in an expanded form, simply because they are too large.
For example, the discriminant of a ternary cubic form, in which case $A=\{z\in \ZZ_{\geq 0}^3~|~z_1+z_2+z_3=3\}$, 
is a degree $12$ polynomial of $2040$ monomials \cite[Chapter 11]{GKZ}.

In this article we are concerned with the discriminant $\AD$ of a quaternary cubic form supported on 
\begin{equation}\label{eq:3Delta3}
\tDt = \{ z \in \ZZ^4_{\geq 0} ~|~ z_1+z_2+z_3+z_4=3 \} \enspace ,
\end{equation}
the set of integral points of the $3$-dilated standard tetrahedron embedded in~$\RR^4$.
Throughout this note $\tDt$ denotes the point configuration defined in \eqref{eq:3Delta3}, whereas $A$ is used for arbitrary point configurations.

The symmetric group on four letters $\Sfour$ acts on $\tDt$ by coordinate permutation.
This action naturally extends to the discriminant $\AD$.
We are concerned with the following question:
\begin{que} \label{question3}
	How many extremal monomials does the discriminant $\AD$ have? 
	Furthermore, how many $\Sfour$-orbits of extremal monomials are there? 
\end{que} 

To answer this question we use combinatorial tools developed in \cite[Chapter 11]{GKZ}.
The key observation of Gel$\cprime$fand, Kapranov and Zelevinsky is that the extremal monomials of the $A$-discriminant are in bijection with so called $D$-equivalence classes of regular triangulations of $A$.
With this at hand, it suffices to compute all regular triangulations of $A$.
The latter task turns out to be a computationally challenging problem, in general.
However, using the recently developed software \mptopcom \cite{mptopcom} we manage to compute all regular triangulations of the point configuration $\tDt$ in question. 
Finally, we efficiently compute the number of $D$-equivalence classes to arrive at our main result.

\begin{thm}\label{thm:main}
	The discriminant $\AD$ has $166\,104$ extremal monomials that split into
	$7\,132$ $\Sfour$-orbits. Equivalently, $\tDt$ admits $166\,104$ $D$-equivalence classes of regular triangulations.
\end{thm} 

Actually, our computation yields a list of exponent vectors of the extremal monomials.
Once these exponent vectors are known, it is also possible to explicitly write down the corresponding coefficients \cite[Theorem 11.3.2]{GKZ}, but we will not be concerned with this.

Representations of the discriminant $\AD$ are known in various fashions.
By definition, $\AD$ equals the resultant of the four quadratic partial derivatives $f_{x_1}$, $f_{x_2}$, $f_{x_3}$, $f_{x_4}$ of a generic polynomial $f$ supported on $\tDt$.
Back in $1899$ Nanson \cite{Nanson} gave a representation of this resultant (and thereby of $\AD$) as the determinant of an $20\times20$ matrix.
Recently $\AD$ was represented as the Pfaffian of an $16\times 16$ matrix \cite{dominicspaper}.
However, with current computer algebra software these matrices are too large to compute their determinants, making this approach to answer Question~\ref{question3} impractical.

Let $P=\mathcal{N}(\AD)$ denote the Newton polytope of the discriminant $\AD$.
Theorem \ref{thm:main} tells us that $P$ has $166\,104$ vertices, split in $7\,132$ $\Sfour$-orbits.
In general, the structure of the Newton polytope of a polynomial is closely related to the geometry of the associated hypersurface. 
For example, the vertices of the Newton polytope are in bijection with the unbounded regions of the amoeba associated to the hypersurface, which can be viewed as a logarithmic shadow of the hypersurface; see \cite[Chapter 6]{GKZ}.
In this way, Theorem \ref{thm:main} sheds light on the discriminantal hypersurface $\nabla_{\tDt}$. 


\section{D-equivalence of regular triangulations}


Let $A=(a_1, \dots , a_n)$ be a finite point configuration in $\R^d$ and let $Q = \conv(A)$ denote its convex hull. 
We always assume that $Q$ has full dimension $\dim(Q)=d$, and that the lattice spanned by $A$ is $\ZZ^d$. That way we do not have to consider the volume with respect to the lattice spanned by $A$, but can rather use the ordinary lattice volume $\vol$ everywhere.

A \emph{triangulation} of $A$ is a collection of simplices $\triang = (\face)_{\face \in \triang}$ with vertices in $A$ such that
\begin{enumerate}
\item $\triang$ is closed under taking faces, i.e. if $\face \in \triang$ and  $\face' < \face$ then also $\face' \in \triang$,
\item $\triang$ covers $Q$, meaning $\bigcup_{\face\in \triang} = Q$, and
\item for any two $\face,\face'\in\triang$ the intersection $\face\cap\face'$ is a face of both $\face$ and $\face'$.
\end{enumerate}
For each triangulation $\triang$ of $A$ we define its \emph{$\gkz$ vector} $\Phi_\triang \in \mathbb{Z}^A$ by
\begin{equation}\label{eq:def_gkz}
\Phi_\triang(a_i) = \sum\limits_{\face \ni a_i} \vol(\face) \enspace ,	
\end{equation}
where the summation is over all maximal simplices of $\triang$ for which $a_i$ is a vertex.
Since $A$ is finite it only admits finitely many triangulations.
Hence the convex hull of all $\gkz$ vectors
\[
\Sigma\text{-poly}(A) = \conv \{\Phi_\triang ~|~ \triang \text{ triangulation of } A\} \subset \R^A \enspace 
\]
is a convex polytope, called the \emph{secondary polytope} of $A$.
A triangulation of $A$ is called \emph{regular} (some prefer \emph{coherent}) if its $\gkz$ vector is vertex of the secondary polytope. 

The theory of regular triangulations and secondary polytopes 
was introduced by Gel'fand, Kapranov and Zelevinsky
\cite{GKZ} and has numerous applications to questions in classic algebraic geometry;
see the monograph of De Loera, Rambau and Santos \cite{Triangulations}.
For instance, it turns out that the secondary polytope of $A$ equals the Newton polytope of the principle $A$-determinant $E_A$ \cite[Chapter 10]{GKZ}.
In some nice cases (for example, in case the associated toric variety $X_A$ is smooth) the principle $A$-determinant can be expressed as 
\begin{equation}\label{eq:factorization}
E_A = \pm \prod_{F < Q} \Delta_{F \cap A} \enspace ,
\end{equation}
where the product is taken over all nonempty faces $F$ of $Q$. 
Here $\Delta_{F \cap A}$ denotes the $(F \cap A)$-discriminant associated to the face $F$. 
In particular, the $A$-discriminant $\Delta_A$ is a factor of $E_A$ since $Q$ itself is also a face. 
Therefore the Newton polytope of the $A$-discriminant, denoted by $\mathcal{N}(\Delta_A)$, is a Minkowski summand of the secondary polytope of $A$,
\begin{equation}\label{eq:minkowski_sum}
	\Sigma\text{-poly}(A) = \mathcal{N}(\Delta_A) + R_A \enspace ,
\end{equation}
where $R_A$ is some rest polytope. 
It follows that for any regular triangulation $\triang$ its $\gkz$ vector $\Phi_\triang$ can be written as
\begin{equation}
	\Phi_\triang = \eta_\triang + r_\triang
\end{equation} 
where $\eta_\triang$ and $r_\triang$ are vertices of $\mathcal{N}(\Delta_A)$ and $P_A$, respectively.
This is a first hint that the vertices $\eta_\triang$ of $\mathcal{N}(\Delta_A)$ can be expressed explicitly in a similar fashion to \eqref{eq:def_gkz}, which we will explain in the following. 
We stick with the notation of \cite[Chapter 11.3]{GKZ}.

Let $\triang$ be a triangulation of $A$. A $j$-dimensional face (simplex) $\face$ of $\triang$ is called \emph{massive} if it is contained in some $j$-dimensional face of $Q=\conv(A)$. In that case the face of $Q$ is unique and we denote it by $\Gamma(\face)$.
In particular, any $d$-face $\face$ of $\triang$ is massive with $\Gamma(\face)=Q$.
Denote by $M_\triang^j(a_i)$ the set of massive $j$-simplices of $\triang$ that have $a_i$ as a vertex. 
We set 
\[
\eta_{\triang,j}(a_i) = \sum \limits_{\face \in M_\triang^j(a_i)} \vol(\face) \enspace.
\] 
Here $\vol(\face)$ is lattice volume of $\face$ in the intersection lattice $\ZZ^d \cap \aff(\face)$.
Note that $\eta_{\triang,d}$ coincides with the $\gkz$ vector $\Phi_\triang$ defined in \eqref{eq:def_gkz}.

\begin{dfn}\label{def:massive_gkz}
We define the \emph{massive $\gkz$ vector} $\eta_\triang \in \mathbb{Z}^A$ of $\triang$ by
\begin{equation}\label{eq:def_massive_gkz}
	\eta_\triang = \sum \limits_{j=0}^{d} (-1)^{d-j} \eta_{\triang,j} \enspace .
\end{equation}
Finally, we say two regular triangulations of $A$ are \emph{$D$-equivalent} if they have the same massive $\gkz$ vector.
\end{dfn}

\begin{thm}{\cite[Chapter 11, Theorem 3.2]{GKZ}}\label{thm:d_equiv_vert}
	Let $\mathcal{N}(\Delta_A)$ be the Newton polytope of the $A$-discriminatnt $\Delta_A$.
	The vertices of $\mathcal{N}(\Delta_A)$ are exactly the points $\eta_\triang$ for all regular triangulations $\triang$ of $A$. 
	Thus, they are in one-to-one correspondence with the D-equivalence classes of regular triangulations of $A$.
\end{thm} 
In view of \eqref{eq:minkowski_sum} this means that a $D$-equivalence class with massive $\gkz$ vector $\eta$ consists of all those regular triangulations whose $\gkz$ vectors have $\eta$ as a unique $\mathcal{N}(\Delta_A)$-summand.

\begin{exa}\label{ex:massive}
	Let $Q$ be the triangle in $\R^2$ with vertices $(0,0), (2,0), (0,2)$ and let $A$ be the set of lattice points contained in $Q$,
	\[
	A = \{ (0,0) , (0,1) , (0,2) , (1,0) , (1,1) , (2,0) \} \enspace .
	\]
	In Figure \ref{fig:2Delta2_triang} we show all $14$ regular triangulations of $A$ together with both their massive $\gkz$ vectors as well as their ordinary $\gkz$ vectors.
	We see that there are five different D-equivalence classes of regular triangulations with associated massive $\gkz$ vectors
	\begin{align*}
		(1,0,1,0,0,1) & = \eta_{\triang_{0}} \enspace , \\
		(1,0,0,0,2,0) & = \eta_{\triang_{1}} = \eta_{\triang_{2}} = \eta_{\triang_{3}} = \eta_{\triang_{4}} \enspace , \\
		(0,0,1,2,0,0) & = \eta_{\triang_{5}} = \eta_{\triang_{6}} = \eta_{\triang_{7}} = \eta_{\triang_{8}}\enspace , \\
		(0,2,0,0,0,1) & = \eta_{\triang_{9}} = \eta_{\triang_{10}} = \eta_{\triang_{11}} = \eta_{\triang_{12}} \enspace , \\
		(0,1,0,1,1,0) & = \eta_{\triang_{13}}  \enspace .	
	\end{align*}
	By Theorem \ref{thm:d_equiv_vert} we know that $\mathcal{N}(\Delta_A)$ has five vertices given by these massive $\gkz$ vectors. 
	Indeed, the $A$-discriminant of a quadratic form
	\[
	f(x,y) = a_{00} + a_{01}y + a_{02}y^2 +  a_{10}x + a_{11}xy + a_{20}x^2	
	\] 
	is well known to be (see \cite[Chapter 13]{GKZ})
	\[
	\Delta_A = a_{00}a_{11}^2 + a_{01}^2a_{20} + a_{02}a_{10}^2 - a_{01}a_{10}a_{11} - 4a_{00}a_{02}a_{20} \enspace .
	\]
	We observe that the exponent vectors of the $A$-dicriminant above are exactly the massive $\gkz$ vectors that we found.
	Here the variables $a_{00}, a_{01}, \ldots , a_{20}$ are ordered lexicographically. 
	In this case all five monomials of $\Delta_A$ correspond to vertices of its Newton polytope $\mathcal{N}_{\Delta_A}$ (a pentagon).
	Figure \ref{fig:2Delta2_secpoly} depicts the secondary polytope of $A$. Here the vertex colors indicate the D-equivalence classes of regular triangulations. 
	
\end{exa}
\begin{figure}\label{fig:2Delta2_triang}
	\centering
	\begin{tikzpicture}[scale=0.6,line cap=round,line join=round,>=triangle 45,x=1.0cm,y=1.0cm]
	\begin{scriptsize}
	\node at (5,-18) {\begin{tikzpicture}
		\draw[fill=black!10] (0,0)--(2,0)--(0,2)--cycle;
		\twoDeltaTwoTriang{}
		\node at (1,-0.5) {$\eta_{\triang_{0}}=(1,0,1,0,0,1)$};
		\node at (1,-1) {$\Phi_{\triang_{0}}=(4,0,4,0,0,4)$};
		\end{tikzpicture}
	};
	\node at (10,-18) {\begin{tikzpicture}
		\draw[fill=yellow!25] (0,0)--(2,0)--(0,2)--cycle;
		\twoDeltaTwoTriang{10/01,11/10,01/11}
		\node at (1,-0.5) {$\eta_{\triang_{13}}=(0,1,0,1,1,0)$};
		\node at (1,-1) {$\Phi_{\triang_{13}}=(1,3,1,3,3,1)$};
		\end{tikzpicture}
	};
	\node at (0,0) {\begin{tikzpicture}
		\draw[fill=red!20] (0,0)--(2,0)--(0,2)--cycle;
		\twoDeltaTwoTriang{00/11}
		\node at (1,-0.5) {$\eta_{\triang_{1}}=(1,0,0,0,2,0)$};
		\node at (1,-1) {$\Phi_{\triang_{1}}=(4,0,2,0,4,2)$};
		\end{tikzpicture}
	};
	\node at (5,0) {\begin{tikzpicture}
		\draw[fill=red!20] (0,0)--(2,0)--(0,2)--cycle;
		\twoDeltaTwoTriang{00/11,01/11}
		\node at (1,-0.5) {$\eta_{\triang_{2}}=(1,0,0,0,2,0)$};
		\node at (1,-1) {$\Phi_{\triang_{2}}=(3,2,1,0,4,2)$};
		\end{tikzpicture}
	};
	\node at (10,0) {\begin{tikzpicture}
		\draw[fill=red!20] (0,0)--(2,0)--(0,2)--cycle;
		\twoDeltaTwoTriang{00/11,10/11}
		\node at (1,-0.5) {$\eta_{\triang_{3}}=(1,0,0,0,2,0)$};
		\node at (1,-1) {$\Phi_{\triang_{3}}=(3,0,2,2,4,1)$};	
		\end{tikzpicture}
	};
	\node at (15,0) {\begin{tikzpicture}
		\draw[fill=red!20] (0,0)--(2,0)--(0,2)--cycle;
		\twoDeltaTwoTriang{00/11,01/11,10/11}
		\node at (1,-0.5) {$\eta_{\triang_{4}}=(1,0,0,0,2,0)$};
		\node at (1,-1) {$\Phi_{\triang_{4}}=(2,2,1,2,4,1)$};
		\end{tikzpicture}
	};	
	\node at (0,-6) {\begin{tikzpicture}
		\draw[fill=blue!15] (0,0)--(2,0)--(0,2)--cycle;
		\twoDeltaTwoTriang{02/10}
		\node at (1,-0.5) {$\eta_{\triang_{5}}=(0,0,1,2,0,0)$};
		\node at (1,-1) {$\Phi_{T_{5}}=(2,0,4,4,0,2)$};	
		\end{tikzpicture}
	};
	\node at (5,-6) {\begin{tikzpicture}
		\draw[fill=blue!15] (0,0)--(2,0)--(0,2)--cycle;
		\twoDeltaTwoTriang{10/01,02/10}
		\node at (1,-0.5) {$\eta_{\triang_{6}}=(0,0,1,2,0,0)$};
		\node at (1,-1) {$\Phi_{\triang_{6}}=(1,2,3,4,0,2)$};
		\end{tikzpicture}
	};
	\node at (10,-6) {\begin{tikzpicture}
		\draw[fill=blue!15] (0,0)--(2,0)--(0,2)--cycle;
		\twoDeltaTwoTriang{02/10,10/11}
		\node at (1,-0.5) {$\eta_{\triang_{7}}=(0,0,1,2,0,0)$};
		\node at (1,-1) {$\Phi_{\triang_{7}}=(2,0,3,4,2,1)$};	
		\end{tikzpicture}
	};
	\node at (15,-6) {\begin{tikzpicture}
		\draw[fill=blue!15] (0,0)--(2,0)--(0,2)--cycle;
		\twoDeltaTwoTriang{10/01,02/10,11/10}
		\node at (1,-0.5) {$\eta_{\triang_{8}}=(0,0,1,2,0,0)$};
		\node at (1,-1) {$\Phi_{\triang_{8}}=(1,2,2,4,2,1)$};
		\end{tikzpicture}
	};
	\node at (0,-12) {\begin{tikzpicture}
		\draw[fill=green!25] (0,0)--(2,0)--(0,2)--cycle;
		\twoDeltaTwoTriang{01/20}
		\node at (1,-0.5) {$\eta_{\triang_{9}}=(0,2,0,0,0,1)$};
		\node at (1,-1) {$\Phi_{\triang_{9}}=(2,4,2,0,0,4)$};
		\end{tikzpicture}
	};
	\node at (5,-12) {\begin{tikzpicture}
		\draw[fill=green!25] (0,0)--(2,0)--(0,2)--cycle;
		\twoDeltaTwoTriang{10/01,01/20}
		\node at (1,-0.5) {$\eta_{\triang_{10}}=(0,2,0,0,0,1)$};
		\node at (1,-1) {$\Phi_{\triang_{10}}=(1,4,2,2,0,3)$};
		\end{tikzpicture}
	};
	\node at (10,-12) {\begin{tikzpicture}
		\draw[fill=green!25] (0,0)--(2,0)--(0,2)--cycle;
		\twoDeltaTwoTriang{01/20,01/11}
		\node at (1,-0.5) {$\eta_{\triang_{11}}=(0,2,0,0,0,1)$};
		\node at (1,-1) {$\Phi_{\triang_{11}}=(2,4,1,0,2,3)$};
		\end{tikzpicture}
	};
	\node at (15,-12) {\begin{tikzpicture}
		\draw[fill=green!25] (0,0)--(2,0)--(0,2)--cycle;
		\twoDeltaTwoTriang{10/01,01/20,01/11}
		\node at (1,-0.5) {$\eta_{\triang_{12}}=(0,2,0,0,0,1)$};
		\node at (1,-1) {$\Phi_{\triang_{12}}=(1,4,1,2,2,2)$};
		\end{tikzpicture}
	};				
\end{scriptsize}	
\end{tikzpicture}
\caption{All $14$ triangulations of the point configuration $A$ from Example \ref{ex:massive} are regular. There are five different D-equivalence classes, distinguished by coloring.}
\end{figure}
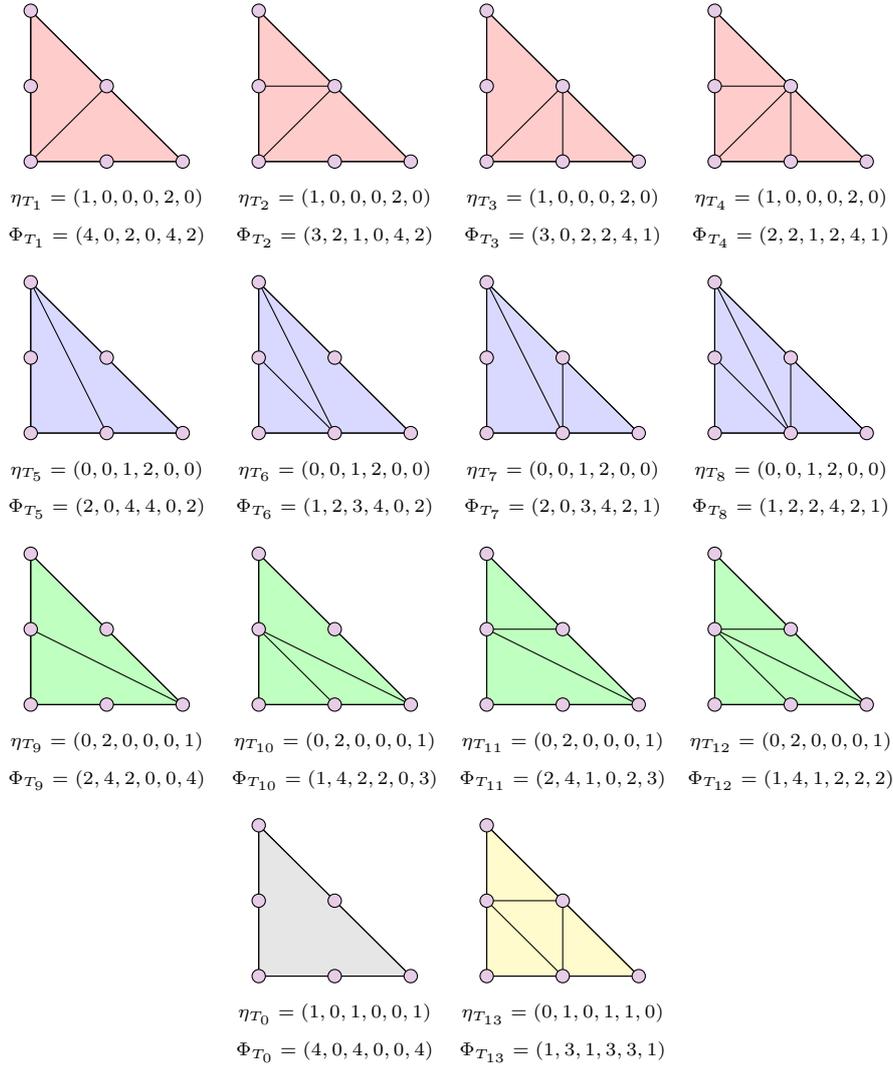

\begin{figure}\label{fig:2Delta2_secpoly}
   \centering
      \includegraphics[scale=0.26]{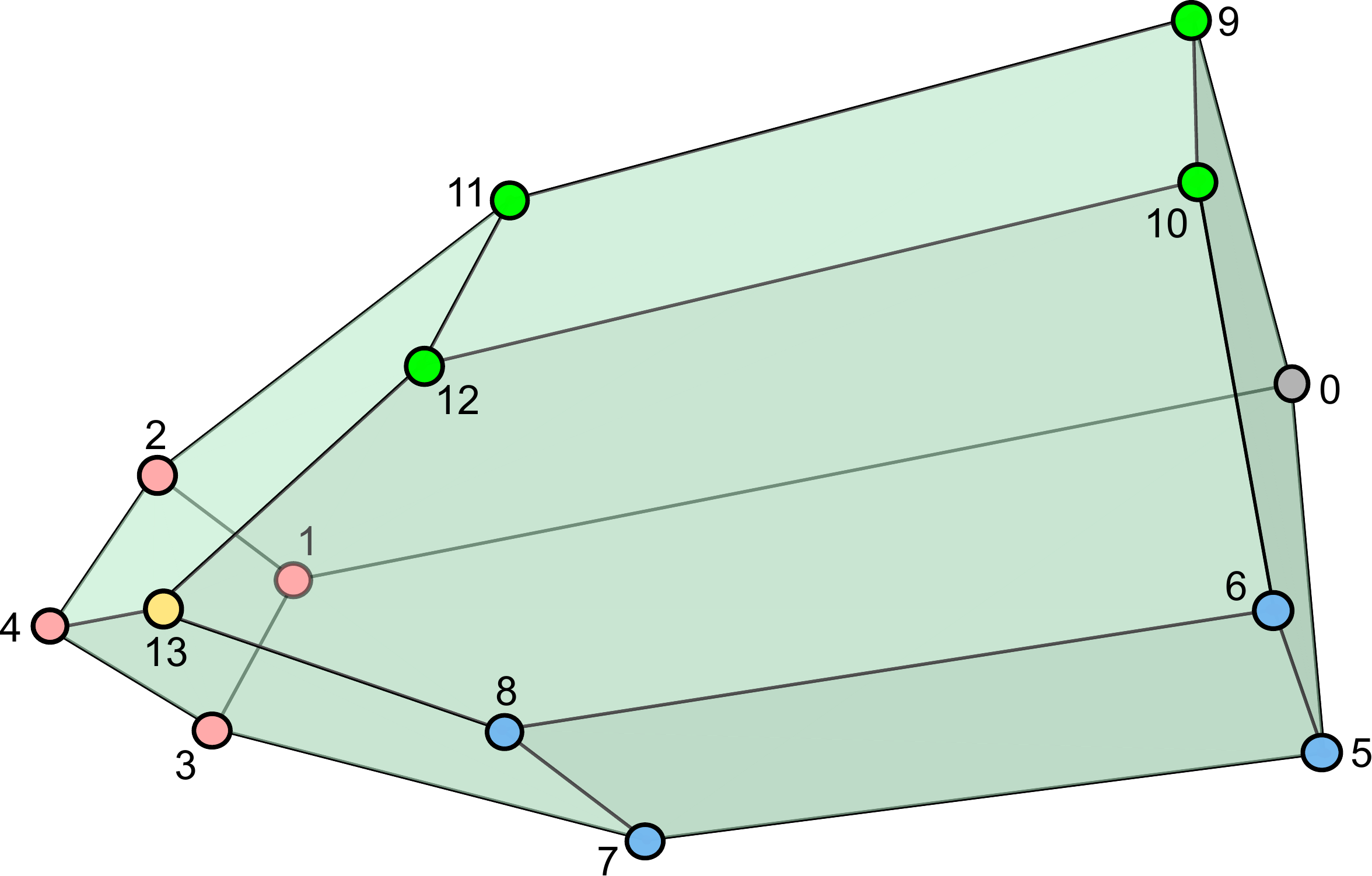}
   \caption{The secondary polytope of the point configuration $A$ from Example \ref{ex:massive} with vertex colors indicating the D-equivalent triangulations.}
\end{figure}

From now on we are only concerned with the point configuration $\tDt$ that supports cubic quaternary forms, i.e. 
\[
\tDt = \{ z \in \ZZ^4_{\geq 0} ~|~ z_1+z_2+z_3+z_4=3 \} \enspace .
\] 
These are the lattice points of the $3$-dilated standard tetrahedron embedded in $\RR^4$.
Our goal is to determine all $D$-equivalence classes of regular triangulations of $A$.
Proceeding as in Example \ref{ex:massive} one way to approach this problem is to first find all regular triangulations of $A$.
This turns out to be a computationally challenging task as the number of regular triangulations explode even for rather ordinary point configurations containing "not too many" points in rather "low" dimensions; see \cite[Chapter 8]{Triangulations}.
However, using the recently developed software framework \mptopcom \cite{mptopcom} we managed to compute all $910\,974\,879$ regular triangulations of $\tDt$ up to $\Sfour$-symmetry (Theorem \ref{thm:num_reg_triang}).
This will be explained later in section \ref{sec:mptopcom}.
We first continue by describing how we use this intermediate result to arrive at the total number of $166\,104$ $D$-equivalence classes of regular triangulations of~$\tDt$ from Theorem \ref{thm:main}.

\section{Massive chains}
The goal is to compute massive $\gkz$ vectors efficiently. There are $910\,974\,879$
regular triangulations of $\tDt$ up to $\Sfour$-symmetry, and we have to
compute the massive $\gkz$ vector for one representative of every orbit. For a
given triangulation $\triang$, the formula given in 
\eqref{eq:def_massive_gkz} involves
\begin{itemize}
\item computing the Hasse diagram of $\triang$ in order to get all faces of $\triang$, 
\item checking which faces of $\triang$ are massive, and
\item computing the lattice volume of the massive faces.
\end{itemize}
Now imagine that this takes $1/10$-th of a second for every triangulation, a
realistic number according to an adhoc implementation in \polymake \cite{polymake}. Then the
overall computation time would amount to $2.9$ years. This computation is
trivially parallelized, but even on $100$ cores it would still take $1.5$
weeks.

By saying that we want to design the computation more efficiently, we mean that
we want to allow longer preprocessing, which is done only once, if in return
the computation of a single massive $\gkz$ vector is sped up. This means that
computation of a single massive $\gkz$ vector might become much slower, whereas
computation of $900$ million massive $\gkz$ vector becomes much faster.

The principle is demonstrated in \mptopcom: For computing the ordinary
$\gkz$ vector, first the volume of any $d$-simplex one can build from the points
in $A$ is computed and stored. Then no volume computation is necessary anymore
throughout the computation, just cache accesses.

The goal of this section is for a given
triangulation $\triang$ to find a representation
\begin{equation}\label{eq:compute_massive_gkz}
	\eta_{\triang}\ =\ \sum_{\face^d\in\triang}\eta(\face^d) \enspace ,
\end{equation}
where the sum runs over the full-dimensional simplices of $\face^d \in
\triang$. Then we can cache the $\eta(\face^d)$ for all possible $\face^d$ that
can occur in $A$ and speed up computations.

Our main tool are massive chains:
\begin{dfn}
Let $\triang$ be a triangulation of a point configuration $A$. 
A sequence $\face^j\le \face^{j+1}\le \ldots
\le \face^{d-1}\le \face^d$ of faces of $\triang$ is called a \emph{massive chain} if for all $k=j,\ldots,d$ we have that $\face^k$ is a massive face of dimension $\dim \face^k\ =\ k$.
\end{dfn}

We say a massive chain \emph{starts in $\face$}, if $\face^j=\face$. This does not mean
that there cannot be a $\face^{j-1}$, it is just a notion for convenience.
Denote by $\MC(\face,\triang)$ the massive chains starting in $\face$ contained in $\triang$.

\begin{exa}
	Consider the triangulation $\triang_1$ from Example \ref{ex:massive}.
	Figure \ref{fig:massive_chains} depicts three sequences of faces of $\triang_1$.
	The first one is not a massive chain since the starting edge is not massive.
	Although all faces of the second sequence are massive faces, it is not a massive chain because the vertex and the triangle differ by two dimensions.
	Finally, the last sequence is a massive chain that starts on a vertex. 
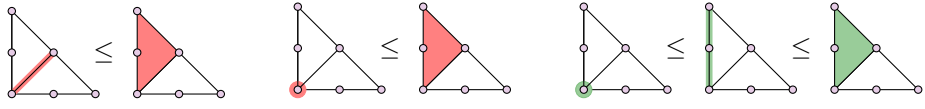
\begin{figure}
	\centering
		\begin{tikzpicture}[scale=0.55]
	\draw (-3,0)--(-1,0)--(-3,2)--cycle;
	\draw [line width=2.5pt, red!50] (-3,0) -- (-2,1);
	\draw (-3,0) -- (-2,1);	
	\draw (-3,0) -- (-3,2);	
	\draw [fill=violet!20] (-3,0) circle (2.5pt);
	\draw [fill=violet!20] (-2,1) circle (2.5pt);
	\draw [fill=violet!20] (-2,0) circle (2.5pt);
	\draw [fill=violet!20] (-1,0) circle (2.5pt);
	\draw [fill=violet!20] (-3,2) circle (2.5pt);	
	\draw [fill=violet!20] (-3,1) circle (2.5pt);
	\node at (-0.8,1) {$\leq$};				
	\draw (0,0)--(1,1)--(2,0)--cycle;	
	\draw [fill=red!50] (0,0)--(1,1)--(0,2)--cycle;
	\draw [fill=violet!20] (0,0) circle (2.5pt);
	\draw [fill=violet!20] (1,1) circle (2.5pt);
	\draw [fill=violet!20] (0,2) circle (2.5pt);
	\draw [fill=violet!20] (0,1) circle (2.5pt);	
	\draw [fill=violet!20] (1,0) circle (2.5pt);	
	\draw [fill=violet!20] (2,0) circle (2.5pt);		
	\end{tikzpicture}
	\qquad
\begin{tikzpicture}[scale=0.55]
\draw (-3,0)--(-1,0)--(-3,2)--cycle;
\draw [red!50,fill=red!50] (-3,0) circle (5.5pt);		
\draw (-3,0) -- (-3,2);	
\draw [] (-3,0) -- (-2,1);
\draw [fill=violet!20] (-3,0) circle (2.5pt);
\draw [fill=violet!20] (-2,1) circle (2.5pt);
\draw [fill=violet!20] (-2,0) circle (2.5pt);
\draw [fill=violet!20] (-1,0) circle (2.5pt);
\draw [fill=violet!20] (-3,2) circle (2.5pt);	
\draw [fill=violet!20] (-3,1) circle (2.5pt);
\node at (-0.8,1) {$\leq$};				
\draw (0,0)--(1,1)--(2,0)--cycle;	
\draw [fill=red!50] (0,0)--(1,1)--(0,2)--cycle;
\draw [fill=violet!20] (0,0) circle (2.5pt);
\draw [fill=violet!20] (1,1) circle (2.5pt);
\draw [fill=violet!20] (0,2) circle (2.5pt);
\draw [fill=violet!20] (0,1) circle (2.5pt);	
\draw [fill=violet!20] (1,0) circle (2.5pt);	
\draw [fill=violet!20] (2,0) circle (2.5pt);		
\end{tikzpicture}
\qquad
	\begin{tikzpicture}[scale=0.55]
\draw (-6,0)--(-4,0)--(-6,2)--cycle;
\draw [mygreen!50,fill=mygreen!40] (-6,0) circle (5.5pt);		
\draw (-6,0) -- (-6,2);	
\draw (-6,0) -- (-5,1);
\draw [fill=violet!20] (-6,0) circle (2.5pt);
\draw [fill=violet!20] (-5,1) circle (2.5pt);
\draw [fill=violet!20] (-5,0) circle (2.5pt);
\draw [fill=violet!20] (-4,0) circle (2.5pt);
\draw [fill=violet!20] (-6,2) circle (2.5pt);	
\draw [fill=violet!20] (-6,1) circle (2.5pt);	
\node at (-3.8,1) {$\leq$};
\draw (-3,0)--(-1,0)--(-3,2)--cycle;
\draw [line width=2.5pt, mygreen!40] (-3,0) -- (-3,2);
\draw (-3,0) -- (-3,2);	
\draw [] (-3,0) -- (-2,1);
\draw [fill=violet!20] (-3,0) circle (2.5pt);
\draw [fill=violet!20] (-2,1) circle (2.5pt);
\draw [fill=violet!20] (-2,0) circle (2.5pt);
\draw [fill=violet!20] (-1,0) circle (2.5pt);
\draw [fill=violet!20] (-3,2) circle (2.5pt);	
\draw [fill=violet!20] (-3,1) circle (2.5pt);
\node at (-0.8,1) {$\leq$};				
\draw (0,0)--(1,1)--(2,0)--cycle;	
\draw [fill=mygreen!40] (0,0)--(1,1)--(0,2)--cycle;
\draw [fill=violet!20] (0,0) circle (2.5pt);
\draw [fill=violet!20] (1,1) circle (2.5pt);
\draw [fill=violet!20] (0,2) circle (2.5pt);
\draw [fill=violet!20] (0,1) circle (2.5pt);	
\draw [fill=violet!20] (1,0) circle (2.5pt);	
\draw [fill=violet!20] (2,0) circle (2.5pt);		
\end{tikzpicture}
\caption{Three sequences of faces of the triangulation $\triang_1$. Only the last one depicts a massive chain.}
\label{fig:massive_chains}
\end{figure}		
\end{exa}

\begin{lemma}\label{lemma:mc}
Every massive face $\face$ of a triangulation $\triang$ is part of some
massive chain. The number of massive chains starting in $\face$ is determined by the smallest face of $Q=\conv(A)$ containing it. 
\end{lemma}
\begin{proof}
We will proceed by induction over $j = \dim\face$.

Firstly let $j=d$. Then $\face$ is a full-dimensional simplex of the triangulation
$\triang$. It forms a massive chain itself, proving the first statement.
Furthermore there is no other chain starting in $\face$, so the number of chains is $1$, even independently of $Q$.

Let us explain the induction step from $j$ to $j-1$. Take $\face^{j-1}$ a massive face of
$\triang$. Then $\face^{j-1}$ is contained in a unique face $F^{j-1}$ of $Q$. 
Now $F^{j-1}$ is the intersection of all the $j$-dimensional faces of
$\conv(A)$ containing it, i.e.
\[
F^{j-1} \ =\ \bigcap_{k=1}^m F^j_k \enspace .
\]
Fix a face $F^j_k$ and consider its triangulation induced by $\triang$.
Then there is a unique $i$-dimensional simplex $\face^j_k \subset F^j_k$ containing the massive face $\face^{j-1}$.
But then $\face^j_k$ is a massive face of $\triang$. Since by induction hypothesis the first
statement is true for $\face^j_k$, we have just proven it for $\face^{j-1}$. 

To proof the second statement let $\#\MC(\face,\triang)$ denote the number of massive chains starting in a face $\face$. 
By the induction hypothesis, for the faces $\face^j_k$ these numbers are determined by the faces
$F^j_k$, so we write $\#\MC(\face^j_k,\triang)=\#\MC(F^j_k)$.
Thus we get the following recursive formula
\[
\#\MC(\face^{j-1},\triang)\ =\ \sum_{k=1}^m\#\MC(\face^j_k,\triang)\ =\ \sum_{k=1}^m\#\MC(F^j_k) \enspace ,
\]
with the recursion anchor $\#\MC(\face^d,\triang)=\#\MC(Q)=1$.  
The last sum runs over all the $j$-dimensional faces of $Q$ containing $F^{j-1}$. Hence it only
depends on $F^{j-1}$, concluding the proof.
\end{proof}

\begin{prop}\label{prop:massive_chain_application}
Fix a dimension $j$, $0\le j\le d$, and a point $a_i\in A$. Let
$\triang$ be a triangulation of $A$ and denote by $M^j_\triang(a_i)$ the set of all massive $j$-dimensional faces of
$\triang$ that have $a_i$ as a vertex. Let $\MC(\face^j,\face^d)$ be the
massive chains starting in $\face^j$ that are contained in a full-dimensional
simplex $\face^d\in\triang$, and define $M^j_{\face^d}(a_i)$ analogously.
Then
\[
\eta_{\triang,j}(a_i)\ =\ \sum_{\face^j\in M^j_\triang(a_i)}\vol(\face^j)\ =\ 
\sum_{\face^d\in\triang}\left(
   \sum_{\face^j\in M^j_{\face^d}(a_i)}\left(
      \frac{\#\MC(\face^j,\face^d)\cdot\vol(\face^j)}{\#\MC(\face^j,\triang)}
   \right)
\right)
\]
\end{prop}
\begin{proof}
First we need to make sure that $\#\MC(\face^j,\triang)\not=0$, but this is
the first statement of Lemma~\ref{lemma:mc}. Now every massive chain in
$\triang$ is contained in a unique full-dimensional simplex $\face^d\in\triang$ by definition. Rewriting
\[
\sum_{\face^j\in M^j_\triang(a_i)}\vol(\face^j)\ =\
\sum_{\face^j\in M^j_\triang(a_i)}\left(\frac{1}{\#\MC(\face^j,\triang)}\cdot\sum_{C\in\MC(\face^j,\triang)}\vol(\face^j) \right),
\]
and inserting that every massive chain is contained in a unique simplex, one
arrives at the desired formula.
\end{proof}
Note that the denominator $\#\MC(\face^j,\triang)$ does not depend on the triangulation $\triang$, but
rather on $Q=\conv(A)$. Hence, let us write $\#\MC(\face^j,Q)$ for this factor.
Then we can define
\[
\eta_i^j(\sigma^d)\ :=\ \sum_{\face^j\in M^j_{\face^d}(a_i)}\left(
      \frac{\#\MC(\face^j,\sigma^d)\cdot\vol(\face^j)}{\#\MC(\face^j,Q)}
   \right)
\]
and we get $\eta_{T,j}(a_i)=\sum_{\sigma^d\in\triang}\eta_i^j(\sigma^d)$ as desired in \eqref{eq:compute_massive_gkz}. The main
advantage is that we can precompute and cache all $\eta_i^j(\sigma^d)$, for any
simplex one can form from the points in $A$. This reduces the effort for
computing a massive $\gkz$ vector of a triangulation to some lookups and
additions. We avoid determining massivity of faces and intersecting with
boundaries completely. Since the triangulations usually vastly outnumber the
simplices we can form from $A$, this will increase performance drastically, if
one needs to compute many massive $\gkz$ vectors.

\begin{exa}\label{ex:massive_chain_application}
Let us use Propostion \ref{prop:massive_chain_application} to recalculate the massive $\gkz$ vector of the traingulation $\triang_1$ from Example \ref{ex:massive}.
There are two full-dimensional simplices in $\triang_1$ which we denote by $\face_L$ and $\face_R$ (left and right).
The massive $\gkz$ vector of $\triang_1$ decomposes as $\eta_{\triang_1}=\eta(\face_L)+\eta(\face_R)$.
Let us compute $\eta(\face_R)$.
It is defined by 
\begin{align*}
\eta(\face_R) & = \eta^0(\face_R) - \eta^1(\face_R) + \eta^2(\face_R) \\
~ & = (1/2,0,0,0,0,1) - (2,0,0,0,1,3) + (2,0,0,0,2,2) \\
~ & = (1/2,0,0,0,1,0) \enspace.
\end{align*}
Let us focus on the sixth (last) entry of this vector corresponding to the bottom right vertex $\face^0=(2,0)$.
As a vertex $\face^0$ is a massive $0$-face of induced lattice volume $\vol(\face^0)=1$. 
Any massive chain that starts at $\face^0$ terminates in the right triangle $\face_R$ and therefore $\eta_6^0(\face_R)=1$.
Furthermore, there are two massive edges in $\triang_1$ incident to the vertex $\face^0$, whose lattice volumes are $1$ and $2$, respectively.
Both edges are uniquely completed to a massive chain by adding the triangle $\face_R$, yielding $\eta_6^1(\face_R)=3$.
Finally, $\eta_6^2(\face_R)=2$ since $\face_R$ has lattice volume $2$.
Similarly we get $\eta(\face_L)=(1/2,0,0,0,1,0)$.
In total we obtain the desired massive $\gkz$ vector $\eta_{\triang_1}=(1,0,0,0,2,0)$.
\end{exa}

\subsection{Implementations and timings}
There are two implementations of massive gkz vectors in \polymake. One is a C++
client using the formula of Definition~\ref{def:massive_gkz} directly. The
other is a simple perl
script using massive chains. The preprocessing time is not directly
extractable, since the cache of vectors $\eta(\face)$ is populated on the fly,
which eliminates the task of constructing all possible simplices on startup.
Since we will always be provided with triangulations, we know that we will only
encounter valid simplices. We ran each of the methods on a batch of
$4\,215\,120$ triangulations, and the resulting times are depicted in
Table~\ref{table:timing}.
\begin{table}
\caption{Timings for computing massive $\gkz$ vectors of a batch of $4\,215\,120$ triangulations via different implementations}\label{table:timing}
\centering
\begin{tabular}{lr}
\toprule
Method & Time in s\\
\midrule
Loop with C++ client (full batch) & 103\,809.90\\
perl script with massive chains (full batch) & 10\,177.12\\
perl script with massive chains (full batch, second run) & 4\,677.97\\
estimated preprocessing time & ca. 6\,000.00 \\
\midrule
Loop with C++ client (1000 triangulations) 
& 40.55\\
perl script with massive chains (1000 triangulations) 
& 1\,137.20\\
\bottomrule
\end{tabular}
\end{table}

The perl script took a little under 3 hours, while the C++ method took about
28 hours, a speedup of roughly factor $10$. For larger batches the speedup
becomes even greater, since the preprocessing time stays constant, as
demonstrated with the second run of the perl script on the same batch, using
the stored cache from the first run. We can use this to estimate the
preprocessing time to be about $6\,000$ seconds, then we get a speedup of more
than factor $20$ for the actual computation. If we extrapolate for $189$
batches, then the C++ method would take almost a year, while the perl script
finishes within $10$ days.

On the contrary look at timings on a small batch of $1\,000$ triangulations in
the last two lines of Table~\ref{table:timing}. Here the perl script is almost
$30$ times slower than the C++ implementation. The timings in
Table~\ref{table:timing} are for one run only and do not depict proper
benchmarkings, but they provide an idea of the strength of using massive
chains. Note that since we populate the cache on the fly, running the perl
script on $1\,000$ triangulations will not populate the full cache.

\section{Computing triangulations with \mptopcom} \label{sec:mptopcom}

In order to compute all regular triangulations of $\tDt$, we used
$\mptopcom$, a software framework for enumerating triangulations in parallel
(\cite{mptopcom, tropical_cubics}). It is based on \polymake (\cite{polymake}),
\topcom (\cite{topcom}) and \mts (\cite{mts}). The core idea of $\mptopcom$ is
the following: The regular triangulations of a point configuration form the
vertices of the secondary polytope and the regular flips are the edges of the
polytope. Ordering the vertices lexicographically we get a so-called
\emph{reverse search structure} (\cite{reverse_search}). In essence a reverse
search structure gives a rooted tree within a graph, such that the membership
of an edge in the tree can be determined locally, i.e. only knowing its
endpoints. This tree is then traversed in a depth first search.

We will not go into detail about the flip graph of triangulations and
triangulation orbits here, but refer the reader to
\cite{Triangulations} and \cite{Imai:2002}. Instead, we will just describe the interesting ingredients around the particular example of $\tDt$, leading to the following result.

\begin{thm}\label{thm:num_reg_triang}
	The point configuration $\tDt = \{ z \in \ZZ^4_{\geq 0} ~|~ z_1+z_2+z_3+z_4=3 \}$ admits $910\,974\,879$ $\Sfour$-orbits of regular triangulations.
\end{thm}

\subsection{Parallelization}
Parallelization in \mptopcom is done via the so-called budgeted reverse search
(\cite{mplrs, mts}). This means that every worker gets a budget, which
basically is the maximum depth that he is allowed to explore in the depth first
search. Unexplored nodes are then returned to the master, who stores them in a
queue, until they can be redistributed among the idle workers. The \mts
framework will dynamically adjust the budget in order to maintain a constant
high load to exhaust the given resources optimally.

\subsection{Exploiting group action}
The symmetric group $\Sfour$ acts on $\tDt$ by coordinate permutations. 
This group action naturally expands to triangulations. 
An important property of this action is that it keeps volumes of simplices
constant. Thus, it acts on $\gkz$ vectors via permutation and $\mptopcom$ can
exploit this group action. Instead of triangulations, we want to consider
orbits of triangulations. In order to do this, it is necessary to be able to
get a canonical representative from the orbit of a triangulation and we pick
the triangulation with the lex-maximal $\gkz$ vector from the same orbit as
canonical representative. Finding this triangulation is equivalent to sorting a
vector descending with a restricted set of permutations, namely those from
$\group$. This problem has been solved effectively in \cite{mptopcom} and
implemented in $\mptopcom$, making it very unlikely that a full orbit ever
needs to be computed.

Even though we are interested in the actual triangulations, rather than just
orbits, it is much more effective to have \mptopcom utilize the group action
and expand orbits later on. For instance, regularity then only needs to be
checked on the canonical representative, vastly improving performance.

\subsection{Checkpointing}
An important feature of \mts is its ability to write checkpoint files and to
restart from these. The principle is to let the single workers finish their
jobs without assigning them new jobs. When all the workers have finished, the
master queue is written to a file. Long lasting computations on big clusters
are prone to a variety of interferences: power and network outages, disc
failures, random restarts, etc. In our concrete case the computation took ca.
$80$ days. It is very unrealistic to expect a resource intensive computation to
be able to run for $80$ days without any disturbance. Thus, we chose to produce
a checkpoint every $12$ hours. Depending on the input parameters it can take a
long time to produce the checkpoint, i.e. for all workers to finish. We
triggered the end of the computation after $8$ hours and then it took ca. $4$
hours for all workers to finish.

\subsection{Estimating result size}
Another advantage of checkpointing is that one can analyse the intermediate
results. In our case we already know that $\tDt$ has $21\, 125\, 102$
regular and full triangulations up to symmetry (\cite{mptopcom}). These
triangulations will reappear as a subset of the regular triangulations. Since
the sum over all entries of any $\gkz$ vector is a constant, and the
$\gkz$ vector of a full triangulation does not have any zeros, we can deduce
that the lexicographically largest $\gkz$ vectors belong to non-full
triangulations. Hence, we expect the full triangulations to appear rather
towards the end of our computation. Still, the budgeted reverse search is
different from reverse search in certain aspects and the curve we observe is
depicted in Figure~\ref{fig:full_curve}.

\begin{figure}
\centering
\includegraphics{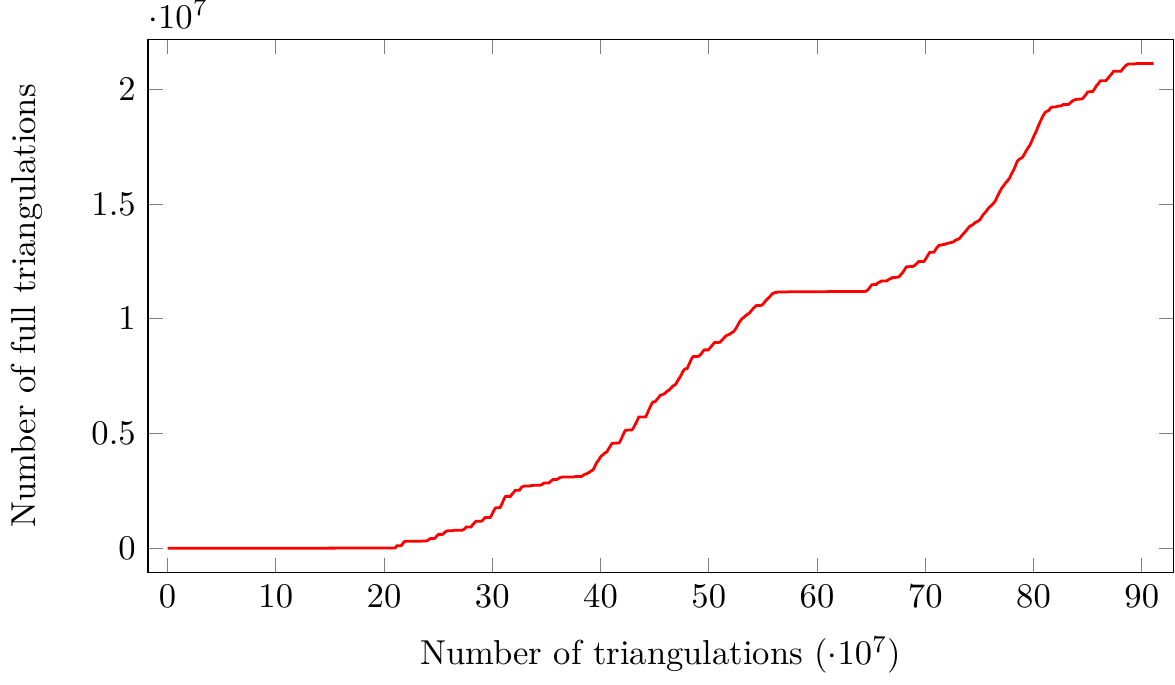}
\caption{Number of full triangulations during computation}\label{fig:full_curve}
\end{figure}

\subsection{Scheduling on large clusters}
Often computations on large clusters are managed via a software scheduler that
one submits jobs to. The scheduler will then arrange those jobs in a way that
optimizes efficiency, i.e. such that idle time is low, jobs with high priority
are executed faster, etc. When submitting a job, one has to specify the
resources needed, like number of nodes, amount of memory, execution time, and
so on. The priority of a job is computed by the scheduler upon submission and
grows with the amount of time that the job has been waiting. Jobs with high
resource demands will get lower priority and will then have to wait until their
priority becomes high enough. This means that if one requires 300 nodes for 220
hours (the maximal available time on the cluster), the job will have to queue
for a long time. On the other hand, if jobs are short, then the scheduler will
often be able to squeeze them in, even if the other resource demands are high.
This is why we chose to make a checkpoint every 12 hours. In total the cluster
ran $189$ jobs on $128$ nodes. This took $6\,892\,460$ seconds, that is
approximately 80 days. On a laptop with four cores, this would be roughly $7$
years.

The result is stored in $189$ compressed files, with a total size of $16.5$
GiB. For compression we used xz \cite{xz}. The uncompressed total size is
$338.2$ GiB.

\subsection{Incomplete result files}
Sometimes it would happen that the number of a triangulations written to the
result file would not agree with the number reported by \mptopcom. There are
several possible reasons for this. Network problems can have prevented a worker
from sending the triangulations to the output worker. The output worker might
have gotten shut down at the end before he managed to write the buffer
completely. When we noticed this, already several subsequent jobs had run. Due
to its parallel nature, the output one gets from one job is not deterministic,
the order of triangulations may vary and also the number. Thus we had to devise
a way to find all triangulations between two checkpoints. One can boil this
down to a problem of graph theory: One has two sets of nodes, a set of starting
nodes and a set of target nodes. To find all nodes between these two sets, one
can run a depth first search from every starting node and never go deeper once
a node belongs to the target set. Back to our original problem, we just needed
to manipulate the \mptopcom source to check for every triangulation that it
found whether it belongs to the target checkpoint, and if it does, exit.

\section{Datasets and code}
\mptopcom can be downloaded at \url{https://polymake.org/mptopcom}. The triangulations of $\tDt$ are available at \url{https://polymake.org/doku.php/dequivalence} as 189 ".xz" files, together with the perl script producing their massive $\gkz$ vectors and the sets of massive $\gkz$ vectors stored as \polymake sets.

\section*{Acknowledgements}
This article adresses the third of the \emph{$27$ questions on cubic surfaces}
\cite{27Q}.  We would like to thank Dominic Bunnett, Michael Joswig, Marta
Panizzut, and Bernd Sturmfels for bringing this problem to our attention and
for helpful discussions.  We are very grateful to Benjamin Lorenz for help and
advice during setting up our software on the TU Berlin clusters and for help
when implementing massive $\gkz$ vectors in \polymake.

\printbibliography

\end{document}